\documentclass[12pt]{amsart}
\usepackage[utf8]{inputenc}
\usepackage{amsmath}
\usepackage{mathtools}
\usepackage{amsthm}
\usepackage{euscript, mathrsfs}
\usepackage{fullpage}
\usepackage{xcolor}
\usepackage{array}
\usepackage{url}
\numberwithin{equation}{section}
\usepackage{marginnote}
\usepackage[shortlabels]{enumitem}
\usepackage{appendix}
\usepackage{lscape}
\usepackage{afterpage}

\renewcommand{\arraystretch}{0.8}

\newtheorem{thm}{Theorem}[section]
\newtheorem{cor}[thm]{Corollary}
\newtheorem{lem}[thm]{Lemma}
\newtheorem{prop}[thm]{Proposition}
\theoremstyle{definition}

\newtheorem{defn}[thm]{Definition}
\newtheorem{rmk}[thm]{Remark}

\newtheorem{conj}[thm]{Conjecture}

\def\nvol{\operatorname{nvol}}

\def\tes{\operatorname{Tes}}

\def\ptes{\operatorname{PTes}}

\def\fcone{\operatorname{fcone}}

\def\ncone{\operatorname{ncone}}

\def\lin{\operatorname{lin}}
\def\diag{\operatorname{diag}}

\def\1{{\boldsymbol{1}}}
\def\0{{\boldsymbol{0}}}
\def\ba{{\boldsymbol{a}}}
\def\bb{{\boldsymbol{b}}}

\def\be{{\boldsymbol{e}}}

\def\bv{{\boldsymbol{v}}}
\def\bw{{\boldsymbol{w}}}

\def\bp{\boldsymbol{p}}
\def\bd{\boldsymbol{d}}
\def\br{\boldsymbol{r}}

\def\bx{\boldsymbol{x}}

\def\bp{\boldsymbol{p}}
\def\bn{\boldsymbol{n}}
\def\by{\boldsymbol{y}}

\def\bu{\boldsymbol{u}}
\def\bp{\boldsymbol{p}}

\def\bvalpha{\operatorname{\alpha^{BV}}}
\def\Hs{\boldsymbol{\eta}}
\def\hs{\eta}

\def\sm{\boldsymbol{\operatorname{\mathit{s}}}}
\def\MM{{\rm{M}}}
\def\XX{{\rm{X}}}
\def\YY{{\rm{Y}}}

\def\U{\operatorname{\mathbb{U}}}
\def\Z{\operatorname{\mathbb{Z}}}
\def\R{\operatorname{\mathbb{R}}}

\def\H{\mathbb{H}}

\makeatletter
\renewcommand*\env@matrix[1][\arraystretch]{%
  \edef\arraystretch{#1}%
  \hskip -\arraycolsep
  \let\@ifnextchar\new@ifnextchar
  \array{*\c@MaxMatrixCols c}}
\makeatother

\title{Ehrhart positivity of Tesler polytopes and Berline-Vergne's valuation}
\author{Yonggyu Lee and Fu Liu}
\date{\today}

\address{Yonggyu Lee, Department of Mathematics, University of California, Davis, One Shields Avenue, Davis, CA 95616 USA.}
\email{ygulee@ucdavis.edu}
\address{Fu Liu, Department of Mathematics, University of California, Davis, One Shields Avenue, Davis, CA 95616 USA.}
\email{fuliu@math.ucdavis.edu}
\begin{document}

\maketitle
\begin{abstract}

  For $\ba \in \R_{\geq 0}^{n}$, the Tesler polytope $\tes_{n}(\ba)$ is the set of upper triangular matrices with non-negative entries whose hook sum vector is $\ba$. 
Motivated by a conjecture of Morales', we study the questions of whether the coefficients of the Ehrhart polynomial of $\tes_n(1,1,\dots,1)$ are positive. 
We attack this problem by studying a certain function constructed by Berline-Vergne and its values on faces of a unimodularly equivalent copy of $\tes_n(1,1,\dots,1).$ 
We develop a method of obtaining the dot products appeared in formulas for computing Berline-Vergne's function directly from facet normal vectors. Using this method together with known formulas, we are able to show Berline-Vergne's function has positive values on codimension $2$ and $3$ faces of the polytopes we consider. As a consequence, we prove that the $3$rd and $4$th coefficients of the Ehrhart polynomial of $\tes_{n}(1,\dots,1)$ are positive. Using the Reduction Theorem by Castillo and the second author, we generalize the above result to all deformations of $\tes_{n}(1,\dots,1)$ including all the integral Tesler polytopes.
\end{abstract}

\section{Introduction}
A subset $P$ of $\R^{m}$ is a \emph{polyhedron} if it is the intersection of finitely many half spaces, usually defined by linear inequalities. A \emph{polytope} is a bounded polyhedron. Equivalently, a polytope $P \subset \R^{m}$ may be defined as the convex hull of finitely many points in $\R^{m}$. We assume that readers are familiar with the basic concepts related to polytopes, such as \emph{face} and \emph{dimension}, presented in \cite{barvinok2008integer, ziegler2012lectures}. 

A polytope $P \subset \R^{m}$ is called \emph{integral} if all of its vertices are integer points, i.e., points in $\Z^m$. In 1962, Ehrhart discovered that for any integral polytope $P$ of dimension $d$, the function $E_{P}$ which maps any non-negative integer $t \in \Z_{\geq 0}$ to the number of integer points in $tP$ (the $t$-th dilate of $P$) is a polynomial in $t$ of degree $d$. We call $E_{P}$ the \emph{Ehrhart polynomial} of $P$. For each $1 \le i \le d,$ let $e_i(P)$ be the coefficient of $t^i$ in $E_P(t),$ so 
\[E_P(t)=e_{d}(P)t^{d}+e_{d-1}(P)t^{d-1}+\cdots+ e_0(P).\]
In \cite{ehrhart1967probleme}, Ehrhart showed that for any $d$-dimensional integral polytope $P,$ the leading coefficient $e_{d}(P)$ of the Ehrhart polynomial $E_P(t)$ is the normalized volume of $P$, the second coefficient $e_{d-1}(P)$ is one half of the sum of the normalized volumes of the facets of $P$, and the constant term $e_0(P)$ is $1$. Thus, these three coefficients are always positive. However, the remaining coefficients of $E_P(t)$ are not always positive. We say an integral polytope $P$ is \emph{Ehrhart positive}, if all the coefficients of $E_P(t)$ are positive. (See \cite{liu2019positivity} for a survey on Ehrhart positivity.)

An affine transformation from an affine space $A \subseteq \R^{m}$ to an affine space $B \subseteq \R^{l}$ is a \emph{unimodular transformation} if it induces a bijection from integer points in $A$ to integer points in $B$. Two polytopes $P \subset \R^{m}$ and $Q \subset \R^{l}$ are said to be \emph{unimodularly equivalent} if there exists a unimodular transformation $\phi$ from 
an affine space containing $P$ to an affine space containing $Q$ such that $\phi(P)=Q$. Clearly, if two polytopes are unimodularly equivalent, they have the same Ehrhart polynomial.

\subsection{Tesler polytopes and Morales' conjecture} 

For any $\ba=(a_1,\dots,a_n) \in \R_{\geq 0}^{n}$, the \emph{Tesler polytope of hook sum $\ba$}, denoted by $\tes_{n}(\ba)$, is the set of all $n \times n$ upper triangular matrices with non-negative entries such that its ``hook sum vector'' is $\ba.$ (See Definition \ref{defn:tesler} for a complete definition.)
When $\ba = \boldsymbol{1}:=(1,\dots,1) \in \R^n,$ the integer points of $\tes_{n}(\ba)$ are called Tesler matrices, which were initially introduced by Tesler,
and then rediscovered by Haglund in his work of expressing the diagonal Hilbert series as a weighted sum over these matrices \cite{haglund2011polynomial}. Consequently, Tesler matrices play an important role in the field of diagonal harmonics \cite{armstrong2012combinatorics, garsia2014constant, gorsky2015refined, haglund2018delta, wilson2017weighted}.
Intrigued by these work, M\'{e}sz\'{a}ros, Morales and Rhoades \cite{Meszaros2017} defined and studied the Tesler polytopes of hook sum $\ba$ defined as above for $\ba \in \Z_{>0}^n$. In this paper, we extend the domain of $\ba$ to $\R_{\geq 0}^{n}$, noticing that important features remain the same.
The work in this paper was initially motivated by a conjecture of Morales':

\begin{conj} \cite{Moralesconjecture} \label{conj}
Tesler polytopes $\tes_{n}(\boldsymbol{1})$ and $\tes_{n}(1,0,\dots, 0)$ are both Ehrhart positive for any positive integer $n$.
\end{conj}

We remark that both families of polytopes in Morales' conjecture are fascinating objects both for their own interesting combinatorial properties and for their connection to other fields of mathematics.
We have already discussed above the importance of $\tes_{n}(\boldsymbol{1})$ in the study of diagonal harmonics. 
Meanwhile, it is known \cite{corteel2017volumes} that $\tes_{n}(1,0,\dots,0)$ is unimodularly equivalent to the Chan-Robbins-Yuen (CRY) polytope, whose volume is the product of the first $n-2$ Catalan numbers \cite{chan2000volume, zeilberger1999proof}. As of today, still no simple proofs for this surprising result is known.
Furthermore, the CRY polytope is a face of the Birkhoff polyope which is also a well-studied subject. In particular, computing volumes of Birkhoff polytopes is an extremely hard problem and has attracted a lot of recent research \cite{canfield2007asymptotic, de2009generating, pak2000four}. 

\subsection{McMullen's formula and $\alpha$-positivity} In this paper, we will use a technique developed by Castillo and the 2nd author \cite{castillo2018berline} to attack Morales' positivity conjecture and its generalizations. The technique is based on the existence of ``McMullen's formula''. 
In 1975, Danilov questioned, in the context of toric varieties, whether it is possible to construct a function $\alpha$ such that for any integral polytope $P \subset \R^{m}$, the following equation holds 
\[|P\cap \mathbb{Z}^{m}|=\sum_{F~:\text{ a non-empty face of }P} \alpha(F,P) \ \nvol(F),\]
where $\nvol(F)$ is the normalized volume of $F$ and $\alpha(F,P)$ only depends on the normal cone of $P$ at $F$ \cite{danilov1978geometry}. 
McMullen was the first to confirm that it is possible to construct such a function $\alpha$ (in a non-constructive way). Hence, we refer to the above formula as \emph{McMullen's formula} \cite{mcmullen1993valuations}. Since $\alpha$ only depends on normal cones and normal cones are invariant under dilations, we obtain the following expression for $e_i(P)$, the coefficient of $t_i$ in $E_P(t),$ as a weighted sum of $\alpha(F,P):$ 
\begin{equation} \label{maccoeffi}
    e_{i}(P)=\sum_{F~:~i\text{-dimensional face of }P}\alpha(F,P)\nvol(F), \text{ for any }0 \leq i \leq \dim(P).
\end{equation}
One sees that, as a consequence of the above formula, if $\alpha(F,P)>0$ for every $i$-dimensional face $F$ of $P$, then $e_{i}(P)$ is positive. Moreover, if $\alpha(F,P)>0$ for every face of $P$, then $P$ is Ehrhart positive. We say a polytope $P$ is \emph{$\alpha$-positive} if $\alpha(F,P)>0$ for every face $F$ of $P$. 

One cannot discuss $\alpha$-positivity without fixing a construction of the $\alpha$ function for McMullen's formula.
Currently, at least three different constructions for $\alpha$ are known; they are given by Pommersheim and Thomas \cite{pommersheim2004cycles}, by Berline and Vergne \cite{berline2007local}, and by Ring and Sch{\"u}rmann \cite{ring2017local}. Following \cite{castillo2018berline}, we will use Berline-Vergne's construction for the $\alpha$ function, and for simplicity we refer to their construction as the \emph{BV-$\alpha$ function}, denoted by $\bvalpha$.

Based on the discussion above, one sees that Ehrhart positivity can be studied through $\alpha$-positivity or more specifically through BV-$\alpha$-positivity, which is the approach we will take. To simplify the calculation of BV-$\alpha$ values, we work on the ``projected Tesler polytope'' $\ptes_{n}(\ba)$, which is unimodularly equivalent to $\tes_{n}(\ba)$ and thus has the same Ehrhart polynomial as that of $\tes_n(\ba).$ (See \S \ref{subsec:ptes} for the definition of $\ptes_{n}(\ba)$.) Therefore, if $\ptes_{n}(\ba)$ is BV-$\alpha$-positive then $\tes_{n}(\ba)$ is Ehrhart positive. The following is one of our main results: 

\begin{thm} \label{ptes-positive}
  Let $n \ge 3$ be a positive integer. Then $\bvalpha(F, \ptes_n(\1))$ is positive for all the codimension $2$ and $3$ faces $F$ of $\ptes_{n}(\boldsymbol{1})$. Therefore, $e_{d-2}(\ptes_{n}(\boldsymbol{1}))$ and $e_{d-3}(\ptes_{n}(\boldsymbol{1}))$ are positive, where $d=\binom{n}{2}$ is the dimension of $\ptes_{n}(\boldsymbol{1})$.
\end{thm}

One benefit of discussing BV-$\alpha$-positivity is that via the Reduction Theorem (Theorem \ref{reduction}) established by Castillo and the second author, one can prove Ehrhart positivity of all deformations of a polytope $P$ by proving that it is BV-$\alpha$-positive. (See Definition \ref{combisoeq} for the definition of deformations.) 
Applying the Reduction theorem together with the result that any invertible affine transformation preserves deformations (Lemma \ref{unimodulardef}), we obtain the following result:

\begin{cor} \label{teslerdef-positiv} 
Let $P$ be a $d$-dimensional integral polytope that is a deformation of $\tes_{n}(\boldsymbol{1})$. Then the following statements are true:
\begin{enumerate}[label={\rm (\arabic*)}]
  \item If $d = \binom{n}{2},$ i.e., $\dim(P)=\dim(\tes_{n}(\boldsymbol{1}))$, then $e_{d-2}(P)>0$ and $e_{d-3}(P)>0$. 
  \item If $d = \binom{n}{2}-1,$ i.e., $\dim(P)=\dim(\tes_{n}(\boldsymbol{1}))-1$, then $e_{d-2}(P)>0$.
\end{enumerate}
\end{cor}

It is known that $\tes_{n}(\ba)$ is a deformation of $\tes_{n}(\boldsymbol{1})$ for all $\ba \in \R_{\geq 0}^{n}$, and $\tes_n(\ba)$ is integral for $\ba \in \Z_{\ge 0}^n.$ Therefore, we are able to generalize the positivity result stated above to all Tesler polytopes, including $\tes_{n}(1,0,\dots,0)$ - the other polytope in Conjecture \ref{conj}.
\begin{cor}\label{cor:tesler}
  Let $\ba \in \Z_{\ge 0}^n$. Then $e_{d-2}(\tes_n(\ba))$ and $e_{d-3}(\tes_n(\ba))$ are both positive (assuming these coefficients exist), where $d = \dim(\tes_n(\ba)).$
\end{cor}
  
Given the results in Corollary \ref{teslerdef-positiv}, it is natural to ask which polytopes are deformations of $\tes_{n}(\boldsymbol{1})$. Please see \cite[Section 4]{teslerv1} for some discussion on this.

	\subsection{Computing BV-$\alpha$ values}
Our method of proving Theorem \ref{ptes-positive} is by computing $\bvalpha(F,\ptes_n(\1))$ for all the codimension $2$ and $3$ faces $F$ of $\ptes_n(\1)$ using known formulas for BV-$\alpha$ values given in Lemmas \ref{cd2} and \ref{cd3}. 
One difficulty of applying these formulas to find the BV-$\alpha$ values for faces $F$ of $\ptes_{n}(\1)$ (or any other polytope) is that we need to compute dot products between the generators of ``pointed feasible cones'' of $\ptes_{n}(\ba)$ at $F$. Since these formulas are provided for pointed feasible cones that are unimodular with respect to projections of standard lattices, one has to be extremely careful to determine their generating rays.

In \S \ref{subsec:simplecomp}, we develop a method for obtaining dot products appearing in formulas similar to those in Lemmas \ref{cd2} and \ref{cd3}, without calculating the generators of pointed feasible cones. Instead we obtain them via computing dot products between the generators of normal cones. (The key result is summarized in Corollary \ref{CMinv}.) This approach simplifies our verification procedure significantly. In general, our method can be applied to any totally unimodular polytope, and it is efficient if we know the inequality descriptions of these polytopes which are often helpful for figuring out their normal cones. Hence, we expect that our result will be useful in computing BV-$\alpha$ values for other families of totally unimodular polytopes.

\subsection*{Organization of the paper}

In section \ref{sec:polyhedra}, we provide background on polyhedra theory.
In Section \ref{sec:bv}, after giving a brief description for the construction of BV-$\alpha$ functions, we develop the aforementioned method that helps to calculate the BV-$\alpha$ values arising from totally unimodular polytopes.
In Section \ref{sec:positivity}, we first formally introduce Tesler polytopes and review relevant results, and then define projected Tesler polytopes $\ptes_n(\1)$ and discuss its properties. By applying the method developed in \S \ref{sec:bv} to $\ptes_{n}(\boldsymbol{1})$, we obtain a proof for Theorem \ref{ptes-positive}. Finally, we prove Corollaries \ref{teslerdef-positiv} and \ref{cor:tesler} using the Reduction Theorem (Theorem \ref{reduction}).

\subsection*{Acknowledgements} The second author is partially supported by a grant from the Simons Foundation \#426756.

\section{Basic definitions and results in polyhedra theory}\label{sec:polyhedra}
In this section, we review terminologies and results related to polytopes/polyhedra that are relevant to this article. 
We assume that $\langle \cdot, \cdot \rangle$ is the dot product on $\R^{m}$,
and let $P \subset \R^m$ be a polytope. If a facet $F$ of $P$ contains a face $G$, we call $F$ a \emph{supporting facet} of $G$. 

Given a finite set of vectors $R=\{\br_1, \br_2, \dots, \br_k\} \subset \R^m$, a \emph{(polyhedral) cone} $K \subset \R^m$ \emph{generated} by $R$ is
\[ K := \left\{ \bx \in \R^m \ : \ \bx = \sum_{i=1}^k c_i \br_i, \ c_i \ge 0\right\}.\]
A cone $K \subset \R^m$ can also be defined by homogeneous linear inequalities (thus it is polyhedral).
A \emph{pointed cone} is a cone that does not contain a line. 
The \emph{polar cone} of a cone $K \subset \R^m$ is the cone:
\[K^{\circ}=\{\by \in W~~:~~\langle \bx,\by \rangle \leq 0, \ \forall \bx \in K\},\] 
where $W$ is the subspace of $\R^m$ spanned by $K$. It is a well-known result that the polar of the polar of a pointed cone is itself.

For any subset $S$ of $\R^m$, let $\lin(S)$ be the translation of the affine span of $S$ to the origin, and let $\R^m/\lin(S)$ be the orthogonal complement of $\lin(S)$. In this paper, $\lin(S)$ is always a \emph{rational} subspace of $\R^m,$ i.e., it is a subspace that can be defined by linear equalities with integer coefficients. 
For any subset or an element $B$ of $\R^m$, we use $B/\lin(S)$ to denote the canonical projection of $B$ onto $\R^m/\lin(S)$. Note that when $\lin(S)$ is a rational subspace, $\Z^m/\lin(S)$ is a lattice in $\R^m/\lin(S).$

A cone $K \subset \R^m/\lin(S)$ is \emph{unimodular with respect to the lattice  $\Z^m/\lin(S)$} if it can be generated by a set of vectors $\{\bd_{1},\dots,\bd_{l}\}$ that can be extended to a basis for the lattice $\Z^m/\lin(S).$ In this case, we call $\bd_{1},\dots,\bd_{l}$ the \emph{primitive generators} for the unimodular cone $K.$ We say $K$ is \emph{unimodular} if it is unimodular with respect to $\Z^m.$ (Note that for a unimodular cone $K$, each of its primitive generators is a \emph{primitive} vector, i.e., an integer vector with the property that the greatest common divisor of its components is $1$.)
\begin{defn}\label{cones}
  Suppose $P \subset \R^m$ is a polytope.
\begin{enumerate}
  \item For any face $F$ of $P,$ the \emph{feasible cone of $P$ at $F$} is:
    \[\fcone(F,P):=\{\by \in \R^m~:~\bx+\epsilon \by \in P~\text{for some}~\epsilon>0\},\]
    where $\bx$ is any interior point of $F$. (It can be shown that the definition does not depend on the choice of $\bx$.)
    The \emph{pointed feasible cone} of $P$ at $F$ is:
    \[\fcone^{\bp}(F,P):=\fcone(F,P)/\lin(F).\]
    \item For any face $F$ of $P$, the \emph{normal cone} of $P$ at $F$ is:
\[\ncone(F, P) := \{\bu \in \lin(P) ~~:~~\langle \bu,\bp_{1} \rangle \geq \langle \bu,\bp_{2} \rangle, \forall \bp_{1} \in F, \forall \bp_{2} \in P \}.\]
The \emph{normal fan} of $P$ is the collection of all normal cones of $P$ at its non-empty faces.
   
\end{enumerate}
\end{defn}

The following is one important connection between normal cones and pointed feasible cones: 
\begin{lem}\cite[Lemma 2.4]{Castillo2015v1} \label{lem:nfcone}
  Suppose $P \subset \R^m$ is a polytope and $F$ is a codimension $k$ face of $P.$ Then both $\fcone^p(F,P)$ and $\ncone(F,P)$ are full-dimensional pointed cones in the $k$-dimensional space $\lin(P)/\lin(F)$. Furthermore, they are poloar to one another, that is,
\begin{equation}\label{eq:nfcone}
	\ncone(F,P)^{\circ}=\fcone^{p}(F,P) \text{ and } \fcone^{p}(F,P)^\circ = \ncone(F,P).
\end{equation}
\end{lem}

Definition \ref{cones} gives one common definition for pointed feasible cones and normal cones. The following lemma provides another way of constructing/defining them, which we state without proof. 

\begin{lem}\label{lem:construct-cones}
  Suppose $P \subset \R^m$ is a polytope and $F$ is a codimension $k$ face of $P.$ 
  \begin{enumerate}
    \item \label{itm:construct-fcone} If $\bv$ is a vertex of $F$, then $\fcone^p(F,P) = \fcone(\bv,P)/\lin(F).$ 

    \item \label{itm:construct-ncone} Assume the supporting facets of $F$ are $\{F_1, \dots, F_k\}$ and for each $i,$ let $\bn_i \in \lin(P)$ be an outer normal vector for the facet $F_i.$ Then $\ncone(F,P)$ is generated by $\bn_1, \dots, \bn_k$.
  \end{enumerate}
\end{lem}

\begin{defn} \label{defn:unimodular}
  We say a polytope $P$ in $\R^{m}$ is \emph{totally unimodular} if it is an integral polytope and $\fcone(\bv,P)$ is unimodular for all the vertices $\bv$ of $P$.
\end{defn}
We remark that any totally unimodular polytope $P$ is a \emph{simple} polytope, that is, each vertex of $P$ is contained in exactly $\dim(P)$ many facets, or equivalently is contained in exactly $\dim(P)$ many edges.

\begin{defn} \label{combisoeq}
  Let $P$ and $Q$ be two polytopes. Then $Q$ is a \emph{deformation} of $P$ if there exists a surjective map $\phi$ from the set of the vertices of $P$ to that of $Q$ and $r_{i,j} \in \R_{\geq0}$ such that $\phi(\bv_{i})-\phi(\bv_{j})=r_{i,j}(\bv_{i}-\bv_{j})$ whenever $\bv_{i}$ and $\bv_{j}$ are adjacent vertices of $P$.
\end{defn}
It is a classical result \cite{postnikov2008faces} that there is an alternative but equivalent way of defining deformations in terms of normal fans. However, it requires considering normal fans of $P$ and $Q$ with respect to a same underlining space, which we do not include in this paper.
We finish this part with the following lemma that will be used in proving Corollary \ref{teslerdef-positiv}.
\begin{lem} \label{unimodulardef}
Let $P$ and $Q$ be two polytopes and $\psi$ is an invertible affine transformation from the affine hull of $Q$ to the affine hull of $P$. Then $Q$ is a deformation of $P$ if and only if $\psi(Q)$ is a deformation of $\psi(P)$.
\end{lem}

\begin{proof}
  For the forward direction, let $\bv , \bw$ be any pair of adjacent vertices of $P$ and $\bv', \bw'$ the corresponding vertices of $Q$. Then by Definition \ref{combisoeq}, there exists $r \in \R_{\ge 0}$ such that $\bv'-\bw'=r(\bv-\bw)$. Since $\psi$ is an affine transformation, we have that $\psi(\bv')-\psi(\bw')=r(\psi(\bv)-\psi(\bw))$. Therefore, $\psi(Q)$ is a deformation of $\psi(P)$. The backward direction can be proven by the same argument.
\end{proof}

\section{Computing Berline-Vergne's $\alpha$-construction} \label{sec:bv}

In the first part of this section, we provide more details on Berline-Vergne's $\alpha$-function construction for McMullen's formula, formulas for computing it, as well as the Reduction Theorem. 
In the second part, we develop an alternative method of obtaining dot products appearing in formulas for BV-$\alpha$ values, which is one of the main results of this paper. This method and the Reduction Theorem are both key ingredients in our proofs for Theorem \ref{ptes-positive} and Corollary \ref{teslerdef-positiv} that will be presented in Section \ref{sec:positivity}.

\subsection{Berline-Vergne's construction} 
In \cite{berline2007local}, Berline and Vergne associate to every rational affine cone $c$ an analytic function $\phi(c)$ on $\R^{m}$ which is recursively defined with respect to the dimension of cones. They show that $\phi$ is a valuation and $\phi(c)$ is analytic near the origin. Then they set $\bvalpha(F,P)$ to be the residue of $\phi(\fcone^{p}(F,P))$ around $0$, and prove that it is a valid $\alpha$-construction for McMullen's formula. 

One sees that this procedure of computing $\bvalpha$ or $\phi$ is complicated. 
Using the valuation property of $\phi$, one can reduce the problem of computing $\bvalpha(F, P)$ to the cases when  $\fcone^{p}(F,P)$ is unimodular with respect to $\Z^m/\lin(F)$. 
However, even for these cases, simple formulas for $\bvalpha(F,P)$ are only known if the dimension of $\fcone^{p}(F,P)$ is at most $3$. Note that it \cite[Example 19.2]{barvinok2008integer} follows immediately from Berline-Vergne's construction that 
\begin{eqnarray*}
	\bvalpha(F,P)=1, & \quad& \text{ when } \dim(\fcone^{p}(F,P))=0, \\
	\bvalpha(F,P)=\dfrac{1}{2}, & \quad& \text{ when } \dim(\fcone^{p}(F,P))=1.
\end{eqnarray*}
We include formulas for faces of codimensions $2$ and $3$ below.
\begin{lem} \cite[Example 19.3]{barvinok2008integer} \label{cd2} 
Let $F$ be a codimension $2$ face of an integral polytope $P \subset \R^m$. Suppose $\fcone^{p}(F,P)$ is a unimodular cone with respect to the lattice $\Z^{m}/\lin(F)$, and $\bu_1$ and $\bu_2$ are its primitive generators. 
Then 
    \[\bvalpha(F,P)=\frac{1}{4}+\frac{1}{12}\biggl{(}\frac{\left \langle \bu_1,\bu_2\right \rangle}{\left \langle \bu_1,\bu_1\right \rangle}+\frac{\left \langle \bu_1,\bu_2\right \rangle}{\left \langle \bu_2,\bu_2\right \rangle}\biggl{)}.\]
\end{lem}

\begin{lem} \cite[Lemma 3.10]{castillo2018berline} \label{cd3} 
  Let $F$ be a codimension $3$ face of an integral polytope $P \subset \R^m$. Suppose $\fcone^{p}(F,P)$ is a unimodular cone with respect to the lattice $\Z^{m}/\lin(F)$, and $\bu_1$, $\bu_2$ and $\bu_3$ are its primitive generators. Then
    \[\bvalpha(F,P)=\frac{1}{8}+\frac{1}{24}\biggl{(}\frac{\left \langle \bu_1,\bu_2\right \rangle}{\left \langle \bu_1,\bu_1\right \rangle}+\frac{\left \langle \bu_1,\bu_2\right \rangle}{\left \langle \bu_2,\bu_2\right \rangle}+\frac{\left \langle \bu_1,\bu_3\right \rangle}{\left \langle \bu_1,\bu_1\right \rangle}+\frac{\left \langle \bu_1,\bu_3\right \rangle}{\left \langle \bu_3,\bu_3\right \rangle}+\frac{\left \langle \bu_2,\bu_3\right \rangle}{\left \langle \bu_2,\bu_2\right \rangle}+\frac{\left \langle \bu_2,\bu_3\right \rangle}{\left \langle \bu_3,\bu_3\right \rangle}\biggl{)}.\]
\end{lem}

\begin{rmk} \label{alpharmk}
	When $\dim(\fcone^{p}(F,P)) \geq 4$, the formulas for $\bvalpha(F,P)$ that one can obtain by directly applying Berline-Vergne's algorithm become way more complicated. However, by the nature of the algorithm, these formulas are still in terms of $\langle \bu_{i}, \bu_{j} \rangle$, the dot products between primitive generators of the unimodular cone $\fcone^p(F,P).$ Therefore, it is important to know the values of these dot products. 
\end{rmk}

Despite the difficulty and complication of obtaining formulas for computing the BV-$\alpha$ values, there is one great benefit of studying the question of Ehrhart positivity via the BV-$\alpha$-positivity approach. In \cite{castillo2018berline}, Castillo and the second author obtained the following theorem - the Reduction Theorem - from the valuation property of $\phi$:

\begin{thm} [Reduction Theorem] \label{reduction}
Suppose $P$ and $Q$ are two integral polytopes in $\R^{m}$. Assume further that $Q$ is a deformation of $P$. Then for any fixed k, if
$\bvalpha(F, P) > 0$ for every $k$-dimensional face $F$ of P, then $\bvalpha(G, Q) > 0$ for every $k$-dimensional
face $G$ of $Q$.
Therefore, BV-$\alpha$-positivity of $P$ implies BV-$\alpha$-positivity of $Q$.
\end{thm}

By the Reduction Theorem, if we show that all the $k$-dimensional faces of an integral polytope $P$ is BV-$\alpha$-positive (which together with (\ref{maccoeffi}) implies that the corresponding Ehrhart coefficient is positive), then the same is true for any integral deformation of $P$.

\subsection{Computing BV-$\alpha$ values of totally unimodular polytopes.} \label{subsec:simplecomp}

Recall that if a polytope $P$ in $\R^{m}$ is totally unimodular, then for any vertex $v$ of $P$, the feasible cone $\fcone(v,P)$ is unimodular. As a consequence, for any face $F$ of $P,$ the pointed feasible cone $\fcone^{p}(F,P)$ is unimodular with respect to the lattice $\Z^{m}/\lin(F)$. (See Remark \ref{rmk:innerprod} below.) 
Thus, we can apply any known BV-$\alpha$ formulas (such as those presented in Lemmas \ref{cd2} and \ref{cd3}) to calculate the BV-$\alpha$ values for $P$. In this part, we present a way to obtain the dot products $\langle \bu_{i}, \bu_{j} \rangle$ appeared in these formulas. As stated in Remark \ref{alpharmk}, these dot products are generally important for computing BV-$\alpha$ values, so we expect our method will be useful when formulas for computing BV-$\alpha$ values for faces of other codimensions become known.

Below is the key lemma of this section. (Recall that for an edge of $P$ connecting two vertices $\bv$ and $\bw$, the \emph{primitive edge direction} from $\bv$ to $\bw$ is the vector $\bd = a(\bw-\bv)$ where $a\in \R_{>0}$ is chosen so that $\bd$ is a primitive vector.)

\begin{lem}\label{innerprod} Let $P$ be a totally unimodular polytope in $\R^{m}$, and assume that $P$ is full-dimensional. Suppose $F$ is a codimension $k$ face of $P$, and its supporting facets are $F_{1},\dots,F_{k}$ (so $F$ is the intersection of $F_{1},\dots,F_{k}$). Further assume that $\bn_i$ is the primitive outer normal vector of $F_{i}$ for each $1 \le i \le k$. (Hence, $\{\bn_1, \dots, \bn_k\}$ is a set of rays that generates $\ncone(F,P).$) 
	
	Fix a vertex $\bv$ of $F$. For each $i,$ let $\bv_{i}$ be the vertex of $P$ that is adjacent to $\bv$ but not in $F_{i}$, and $\bd_{i}$ the primitive edge direction from $\bv$ to $\bv_{i}$. (Note that the uniqueness of the choice of $\bv_i$ follows from the fact that $P$ is simple.) Define $\bu_{i}:=\bd_{i}/\lin(F)$ to be the projection of $\bd_i$ onto the orthogonal complement of $\lin(F).$ Then the following statements are true:
\begin{enumerate}[label={\rm (\arabic*)}]
   
    \item $\langle \bn_{i}, \bu_{i} \rangle = -1$ for each $1 \le i \le k.$
    \item $\langle \bn_{j}, \bu_{i} \rangle = 0$ for each pair of distinct $i$ and $j$ in $\{1, 2, \dots, k\}$.
    \item\label{itm:genfcone} $\{\bu_{1},\dots,\bu_{k}\}$ generates $\fcone^{p}(F,P)$.
    \item\label{itm:basis} $\{\bu_{1},\dots,\bu_{k}\}$ is a basis for the lattice $\Z^m/\lin(F).$
\end{enumerate}
\end{lem}

\begin{rmk}\label{rmk:innerprod}
  Note that conclusions \ref{itm:genfcone} and \ref{itm:basis} of the lemma above imply that $\fcone^p(F,P)$ is a unimodular cone with respect to the lattice $\Z^m/\lin(F)$, and $\bu_1, \dots, \bu_k$ are its primitive generators. 
\end{rmk}

\begin{proof}[Proof of Lemma \ref{innerprod}]
  Clearly $F_1, \dots, F_k$ are supporting facets of $\bv.$ Since $P$ is simple and full-dimensional in $\R^m$, the vertex $\bv$ is supported by $m$ facets. Assume $F_{k+1}, \dots, F_m$ are the other $m-k$ supporting facets of $v.$ For each $k+1 \le i \le m,$ we similarly let $\bn_i$ be the primitive outer normal vector of $F_{i}$, let $\bv_{i}$ be the vertex of $P$ that is adjacent to $\bv$ but not in $F_{i}$, and $\bd_{i}$ the primitive edge direction from $\bv$ to $\bv_{i}$. By the choices of these vectors and the fact that $P$ is totally unimodular, we have the following:
  \begin{enumerate}[label=(\alph*)]
    \item $\{ \bd_1, \dots, \bd_m\}$ forms a basis for the lattice $\Z^m.$ 
    \item $\{\bd_1,\dots, \bd_m\}$ generates $\fcone(\bv,P) = \fcone^p(\bv,P)$.
  \item  $\{\bn_1, \dots, \bn_m\}$ generates $\ncone(\bv,P)$.
  \item For each $k+1 \le j \le m,$ the vertex $\bv_j \in F$ and thus $\bd_j \in \lin(F)$.
\end{enumerate}
By Lemma \ref{lem:construct-cones}/\eqref{itm:construct-fcone}, we have that $\fcone^p(F,P) = \fcone(\bv,P)/\lin(F)$. This together with (a), (b) and (d) imply conclusions \ref{itm:genfcone} and \ref{itm:basis} of the lemma.

Before we prove conclusions (1) and (2) of the lemma, we prove the following claim first: for each $1 \le i \le m,$
\[ (\rm{i}) \ \langle \bn_{i},\bd_{i} \rangle = -1, \quad \text{and} \quad (\rm{ii}) \ \langle \bn_{j},\bd_{i} \rangle = 0~~\text{if $j\neq i$}.\]
By Lemma \ref{lem:nfcone}, we have that $\fcone^p(\bv,P)$ and $\ncone(\bv,P)$ are polar to one another. Therefore, we immediately have (ii), and that $\langle \bn_{i},\bd_{i} \rangle = -a_i$ for some positive integer $a_i >0$. However, since $\bn_i$ is primitive, by elementary number theory, there exists an integer point/vector $\bx \in \Z^{m}$ such that $\langle \bn_{i}, \bx \rangle=-1.$ But by (a), the vector $\bx$ is a $\Z$-linear combination of $\bd_1, \dots, \bd_m.$ This together with (ii) implies that $a_i = 1$. So (i) holds. 

Now let $i = 1, 2, \dots k.$ Since $\bu_i = \bd_i/\lin(F),$ we have that $\bu_i - \bd_i \in \lin(F)$ which is contained in $\lin(F_j)$ for each $1 \le j \le k.$ Hence, $\langle \bn_j, \bu_{i}-\bd_{i} \rangle = 0$, or equivalently, 
\[ \langle \bn_{j}, \bu_{i} \rangle = \langle \bn_{j}, \bd_{i} \rangle, \quad \text{for any pair of $i$ and $j$ in $\{1, 2, \dots, k\}$}.\]
Therefore, conclusions (1) and (2) follow from (i) and (ii).
\end{proof}

\begin{cor} \label{CMinv}
  Suppose $F, P$, $\{\bn_1, \dots, \bn_k\}$ and $\{\bu_1, \dots, \bu_k\}$ are as given in Lemma \ref{innerprod}. Let $C=(c_{i,j})$ be the $k \times k$ matrix whose $(i,j)$-th entry is $c_{i,j}=\langle \bn_{i}, \bn_{j} \rangle$ and $M=(m_{i,j})$ be the $k \times k$ matrix whose $(i,j)$-th entry is $m_{i,j}=\langle \bu_{i}, \bu_{j} \rangle$. Then $C$ and $M$ are inverse to one another.
\end{cor}

\begin{proof}
	Let $X= (\bn_1, \dots, \bn_k)$ be the $m \times k$ matrix whose columns are $\bn_i$'s, and $Y = (\bu_1, \dots, \bu_k)$ be the $m \times k$ matrix whose columns are $\bu_i$'s.
	Then by parts (1) and (2) of Lemma \ref{innerprod}, we have \[ X^t Y = - I.\]
	By Lemma \ref{lem:nfcone}, the normal cone $\ncone(F,P)$ is full-dimensional in $\R^m/\lin(F).$ Hence, its generating set $\{\bn_1, \dots, \bn_k\}$ is a basis for $\R^m/\lin(F)$. Therefore, for each $1 \le j \le k,$ we have 
	\[ \bu_j = a_{1,j} \bn_1 + \cdots + a_{k,j} \bn_k\] for some real numbers $a_{1,j}, \dots, a_{k,j}.$ 
	Now applying Lemma \ref{innerprod}/(1)(2) again, we get
	\begin{equation} \label{method2}
		\langle \bu_{i}, \bu_{j} \rangle=\langle \bu_i, a_{1,j}\bn_{1}+\cdots+ a_{i,j}\bn_i + \cdots+a_{k,j}\bn_{k}\rangle = -a_{i,j}.
\end{equation}
Hence,	
	\[ Y=  (\bu_1, \dots, \bu_k) =(\bn_1, \dots, \bn_k) (a_{i,j})_{1 \le i, j\le k}  =  X \left(-\langle \bu_{i}, \bu_{j} \rangle\right)_{1 \le i, j\le k}=  X (-M) = - XM.\]
	Therefore,
\[ -I = X^t Y = X^t (-XM) = - (X^t X) M.\] 
Clearly, we have $X^t X = C$. Thus, $C M = I$.
\end{proof}

For convenience, we call the matrix $C$ and the matrix $M$ appearing in the above corollary the \emph{matrix of dot products (MDP}) of $\ncone(F,P)$ and $\fcone^p(F,P)$, respectively. Then We can restate Corollary \ref{CMinv} as the following:

\begin{cor}\label{CMinv-restate}
  Suppose $P$ is a totally unimodular polytope and is full-dimensional in $\R^m.$, and $F$ is a face of $P$. Then the MDP of $\ncone(F,P)$ and the MDP of $\fcone^p(F,P)$ are inverse to one another.
\end{cor}

\begin{rmk}\label{rmk:fulldim}
  We remark that results given in Lemma \ref{innerprod} and its corollaries are stated for totally unimodular polytopes that are full-dimensional in $\R^m$. If a totally unimodular polytope $P$ in $\R^m$ is not full-dimensional, one can consider a unimodularly equivalent copy of $P$ that is full-dimensional in its ambient space and apply our results. This is the approach we will take when we discuss Tesler polytopes in the next section.

  We want to note that it is possible to modify the choice of generators of $\ncone(F,P)$ in Lemma \ref{innerprod} to obtain results that work for totally unimodular polytopes that are not full-dimensional. But since the setup would be more complicated, we only include the current simpler version that we use in this article.  
\end{rmk}

\section{Positivity of the (projected) Tesler polytope of hook sums $(1,\dots,1)$} \label{sec:positivity}

In this section, we start by formally defining Tesler polytopes $\tes_n(\ba)$ and reviewing known results on them. We then define projected Tesler polytopes $\ptes_n(\ba)$ and discuss their facet normal vectors. Next, we apply the method developed in \S \ref{subsec:simplecomp} together with Lemmas \ref{cd2} and \ref{cd3} to show that the BV-$\alpha$ values of all the codimension $2$ and $3$ faces of $\ptes_{n}(\boldsymbol{1})$ are positive. Finally, as a consequence, we complete the proof for our main results - Theorem \ref{ptes-positive} and Corollaries \ref{teslerdef-positiv} and \ref{cor:tesler}. We always assume that $n$ is a positive integer.

\subsection{Background on Tesler polytopes}

In this part, we review definitions and results related to Tesler polytopes that are relevant to this paper. Majority of them are given in \cite{Meszaros2017}.  Let $\U(n)$ be the set of $n \times n$ upper triangular matrices. If we ignore the zeros below the diagonal, then $\U(n) \cong \R^{\tbinom{n+1}{2}}$ as a vector space. 

\begin{defn}\label{defn:tesler}
Let $\MM=(m_{i,j}) \in \U(n)$.
The \emph{$k$-th hook-sum of $\MM$} is the sum of the entries on the $k$-th row of $\MM$ minus the sum of the entries on the $k$-th column of $\MM$ above the diagonal element, that is,
\[\hs_k(\MM) :=(m_{k,k}+m_{k,k+1}+\cdots+m_{k,n}) - (m_{1,k}+m_{2,k}+\cdots+m_{k-1,k}).\]
The \emph{hook-sum vector of $\MM$} is defined to be 
$\Hs(\MM) :=(\hs_1(\MM), \hs_2(\MM),\dots, \hs_n(\MM)).$

For any $\ba=(a_1,\dots,a_n) \in \R_{\geq 0}^{n}$, let
\begin{equation}\label{eq:defn-Hn}
  \H_n(\ba):=\{\MM \in \U(n)~|~\Hs(\MM)=\ba  \},
\end{equation}
be the affine subspace of $\U(n)$ defined by the hook sum conditions determined by $\ba,$ and then define the \emph{Tesler polytope of hook sum $\ba$} to be
\begin{equation}
  \tes_{n}(\ba) := \{\MM=(m_{i,j}) \in \H_n(\ba)~~|~~m_{i,j} \geq 0 \text{ for $1 \le i \le j \le n$} \}. 
	\label{eq:defntesler}
\end{equation}
\end{defn}

\begin{lem}\label{lem:Tes-dim}
  Let $\ba=(a_1, a_2,\dots, a_n) \in \R_{\geq 0}^{n}$. Suppose that the first $p$ entries of $\ba$ are zero and $a_{p+1} > 0,$ and let $\bb = (a_{p+1}, a_{p+2}, \dots, a_n)$, then $\tes_{n}(\ba)$ is isomorphic to $\tes_{n-p}(\bb)$, which is a non-empty polytope of dimension $\binom{n-p}{2}.$
  
  Furthermore, if $a_1 > 0$, i.e., $\ba \in \R_{>0} \times \R_{\ge 0}^{n-1},$ then $\tes_n(\ba)$ is full-dimensional in $\H_n(\ba).$
  
\end{lem}

It was shown in \cite[Corollary 2.6]{Meszaros2017} that the dimension of $\tes_{n}(\ba)$ is $\binom{n}{2}$ for $\ba \in \Z_{>0} \times \Z_{\ge 0}^{n-1}.$ Even though the proof works for $\ba \in \R_{>0} \times \R_{\ge 0}^{n-1}$ as well, their arguments are quite involved. Hence, we provide an alternative and more straightforward proof below. We need the following result on $\H_n(\ba)$, which follows from elementary linear algebra.

\begin{lem}\label{lem:Hn-dim}
  For any $\ba \in \R_{\ge 0}^n,$ the affine space $\H_n(\ba)$ has dimension $\binom{n}{2}.$
\end{lem}

\begin{proof}[Proof of Lemma \ref{lem:Tes-dim}]
  First, $\tes_n(\ba)$ is clearly bounded, and hence is a polytope. Next, it is easy to check that $\tes_n(\ba)$ is isomrphic to $\tes_{n-p}(\bb).$ Given these and that $\tes_n(\ba)$ lives in $\H_n(\ba)$ (which has dimension $\binom{n}{2}$ by Lemma \ref{lem:Hn-dim}), one sees that it suffices to show that if $a_1 > 0,$ then $\tes_n(\ba)$ has dimension $\binom{n}{2}.$

  Assume $a_1 > 0$. We first show that $\tes_n(\ba)$ contains a matrix $\XX=(x_{i,j})$ with all positive entries by construction: 
  \begin{enumerate}
    \item Let $c_1 = a_1/n$ and let $x_{1,j} = c_1$ for all $1 \le j \le n.$
    \item Let $c_2 = (c_1 + a_2)/(n-1)$, and let $x_{2,j} = c_2$ for all $2 \le j \le n.$
    \item Let $c_3 = (c_1+c_2+a_3)/(n-2)$, and let $x_{3,j} = c_3$ for all $3 \le j \le n.$
    \item[(\dots)] \dots \dots \dots
    \item[(n)] Let $c_n = (c_1+c_2 + \cdots + c_{n-1} + a_n)/1$, and let $x_{n,n} = c_n.$
  \end{enumerate}
  It is easy to verify that $\XX=(x_{i,j}) \in \tes_n(\ba)$ with all positive entries, as desired. 

  Now observe that for any $\MM=(m_{i,j}) \in \H_n(0,0,\dots,0)$ with each $|m_{i,j}|$ sufficiently small, we have that $\XX + \MM \in \tes_n(\ba).$ By Lemma \ref{lem:Hn-dim}, the vector space $\H_n(0,0,\dots,0)$ has dimension $\binom{n}{2}.$ Hence, we conclude that $\tes_n(\ba)$ is of dimension $\binom{n}{2}.$
\end{proof}

For any $\ba \in \Z_{\geq 0}^n$, M\'{e}sz\'{a}ros, Morales and Rhoades \cite{Meszaros2017} gave the characterization for the face poset of $\tes_{n}(\ba)$ using the concept of support. 
Their characterization can be easily generalized to any $\ba \in \R_{\geq 0}^{n}$ by the same proof.

\begin{defn}
For $1 \leq i \leq j \leq n$, let 
$H_{i,j}^n$ be the hyperplane consisting of upper triangular matrices whose $(i,j)$-th entry is $0,$ that is,
\begin{equation} \label{hyperplane}
H_{i,j}^{n}:=\{\MM=(m_{l,k}) \in \U(n)~~|~~m_{i,j}=0\}. 
\end{equation}
We say the intersection $H_{i_{1},j_{1}}^{n} \cap \cdots \cap H_{i_{k},j_{k}}^{n}$ \emph{does not make any zero rows} if the intersection is not contained in $H_{i,i}^{n} \cap H_{i,i+1}^{n} \cap \cdots \cap H_{i,n}^{n}$ for every $1 \leq i \leq n$.
\end{defn}
\begin{thm} \cite[Theorem 2.5]{Meszaros2017}\label{Tesler} 
	Let $\ba \in \R_{>0}^n$. If F is a codimension $k$ face of $\tes_n(\ba)$, then $F$ is of the form 
	    \begin{equation}
    \tes_n(\ba) \cap H_{i_{1},j_{1}}^{n} \cap \cdots \cap H_{i_{k},j_{k}}^{n},
		    \label{eq:faceform}
	    \end{equation}
		where $H_{i_{1},j_{1}}^{n} \cap \cdots \cap H_{i_{k},j_{k}}^{n}$ does not make any zero rows. Conversely, any such form is a codimension $k$ face of $\tes_{n}(\ba)$.

\end{thm}

We remark that when we apply Theorem \ref{Tesler}, the hyperplane $H_{n,n}^n$ should never appear in the expression \eqref{eq:faceform} as it automatically makes a zero row. Hence, $\tes_n(\ba) \cap H_{n,n}^n$ is not a facet of $\tes_n(\ba).$ On the other hand, it is easy to see that if $(i,j) \neq (n,n)$, the intersection $\tes_n(\ba) \cap H_{i,j}^n$ is a facet. Hence, we have the following result as a consequence to Theorem \ref{Tesler}.
\begin{cor}\label{cor:Tesler-ineq}
  Let $\ba \in \R_{>0}^n.$ Then the Tesler polytope $\tes_n(\ba)$ has the following inequality description:
\begin{equation*}
\tes_{n}(\ba) = \biggl\{\MM=(m_{i,j}) \in \H_n(\ba)~~\biggl{|}~~m_{i,j} \geq 0\text{  for } {1 \le i < j \le n \text{ or } 1 \leq i = j \leq n-1} \biggl\},
\end{equation*}
in which each inequality defines a facet.
\end{cor}

The following lemma states well-known results about the Tesler polytopes. However, we are not able to find a direct reference for it. So we provide a sketch of its proof without introducing all the terminologies and results involved. 
\begin{lem}\label{lem:tes-totaluni}
Let $\ba \in \R_{\ge 0}^n$. Then the Tesler polytope $\tes_n(\ba)$ is integral if and only if $\ba \in \Z_{\ge 0}^n$.
Furthermore, if $\ba \in \Z_{>0}^n$, the Tesler polytope $\tes_n(\ba)$ is a totally unimodular polytope.
\end{lem}

\begin{proof}[Sketch of Proof]
Note that if we treat $\MM$ as a $\binom{n+1}{2}$-dimensional vector ${\bf M}$, the hook sum condition $\Hs(\MM) =\ba$ can be written as $L {\bf M} = \ba,$ where $L=(l_{k,(i,j)})$ is an $n \times \binom{n+1}{2}$ matrix obtained from the incidence matrix of a directed complete graph on $n$ vertices by appending an identity matrix of size $n$. It is a classical result \cite[Chapter 19]{Schrijver86} that such a matrix is totally unimodular. Hence, our first conclusion follows.

Next, for any $\ba \in \Z_{>0}^{n}$, it follows from \cite[Theorem 1.7]{Meszaros2017} that the Tesler polytope $\tes_{n}(\ba)$ is simple. This together with the fact that $L$ is totally unimodular leads to the second conclusion.
\end{proof}

\subsection{Projected Tesler polytopes and facet normal vectors} \label{subsec:ptes}

We start by defining the projection map we will use in this part. 
For any upper triangular $\XX =(x_{i,j}) \in \mathbb{U}(n)$, we define $\psi_{\diag}(\XX)$ to be the upper triangular matrix $\YY \in \mathbb{U}(n-1)$ obtained by ``erasing'' the diagonal line of $\XX$ shown as below:
	\[\XX = \begin{bmatrix}
	    x_{1,1} & x_{1,2} & x_{1,3} & \cdots & x_{1,n-1} & x_{1,n} \\
	    &  x_{2,2} & x_{2,3} & \cdots &x_{2,n-1} & x_{2,n}\\
        &   &\ddots& \ddots & \vdots & \vdots \\
	&   &  & \ddots & \vdots & \vdots \\
	&   &  & & x_{n-1,n-1}& x_{n-1,n}\\
	&    &    &  &      & x_{n,n}
      \end{bmatrix}\xmapsto{\ \psi_{\diag} \ } 
\begin{bmatrix}
x_{1,2} & x_{1,3} & \cdots & x_{1,n} \\
      & x_{2,3} & \cdots & x_{2,n}\\
      &       & \ddots& \vdots \\
      &       &       &x_{n-1,n}\\
\end{bmatrix}=\YY.\]
More formally, the entries of $\psi_{\diag}(\XX) = \YY=(y_{i,j})$ are defined as
	\begin{equation}
	  y_{i,j} = x_{i,j+1} \quad \text{ for } 1 \le i \le j \le n-1. 
\label{eq:psidiag}
	\end{equation}

Now we are ready to define \emph{projected Tesler polytopes}:
\begin{defn}\label{defn:psidiag}
  For any $\ba \in \R_{\ge 0}^n,$ we call $\psi_{\diag}(\tes_n(\ba))$ the \emph{projected Tesler polytope of hook sum $\ba$} and denote it by $\ptes_n(\ba)$.
\end{defn}

\begin{lem} \label{projmap}
  Let $\ba=(a_{1},\dots,a_{n}) \in \R_{\geq 0}^n$.
  Then the restriction map $\psi_{\diag}|_{\H_n(\ba)}$ is a unimodular transformation from $\H_n(\ba)$ to $\mathbb{U}(n-1)$. (Recall $\H_n(\ba)$ is defined in \eqref{eq:defn-Hn}.)

Therefore, $\ptes_n(\ba)$ is unimodularly equivalent to $\tes_n(\ba).$
\end{lem}

The proof of the above lemma is elementary, and thus is omitted from this paper. One can find a proof of it in \cite[Theorem 3.1.1]{thesis}.

Recall that when $\ba \in \R_{>0}^n,$ for any $(i,j) \neq (n,n)$, the intersection of $\tes_{n}(\ba)$ and the hyperplane $H_{i,j}^{n}$ (which is defined in \eqref{hyperplane}) is a facet of $\tes_n(\ba)$. From now on, we fix a notation for the facet of $\ptes_{n}(\ba)$ that corresponds to the facet $\tes_n(\ba) \cap H_{i,j}(n)$ of $\tes_n(\ba)$ under the transformation $\psi_{\diag}$:
\begin{equation} \label{facetdef}
  F_{i,j}(\ba):=\psi_{\diag}(\tes_{n}(\ba) \cap H_{i,j}^{n}),
\end{equation}
Using Lemma \ref{projmap}, we can easily translate Lemmas \ref{lem:Tes-dim} and \ref{lem:tes-totaluni}, and Theorem \ref{Tesler} to a version for projected Tesler polytopes as below:

\begin{prop}\label{prop:ptes}
  Let $\ba =(a_1, \dots, a_n) \in \R_{\ge 0}^n$. Then the following statements are true.
\begin{enumerate}
  \item \label{itm:pdim} 
    If $a_1 > 0$, i.e., $\ba \in \R_{>0} \times \R_{\ge 0}^{n-1},$ then $\ptes_n(\ba)$ is a full-dimensional polytope in $\U(n-1).$
  \item\label{itm:pint} The projected Tesler polytope $\ptes_n(\ba)$ is integral if and only if $\ba \in \Z_{\ge 0}^n.$
  \item\label{itm:ptot} If $\ba \in \Z_{> 0}^n$, then the projected Tesler polytope $\ptes_n(\ba)$ is totally unimodular. 

  \item\label{itm:pface} If $\ba \in \R_{>0}^n,$ then for any $k$ facets $F_{i_{1},j_{1}}(\ba), \cdots, F_{i_{k},j_{k}}(\ba)$ of $\ptes_n(\ba)$, they are precisely the $k$ supporting facets of a codimension $k$ face of $\ptes_n(\ba)$ if and only if $H_{i_{1},j_{1}}^{n} \cap \cdots \cap H_{i_{k},j_{k}}^{n}$ does not make any zero rows. 
\end{enumerate}
\end{prop}

Similarly, we can convert Corollary \ref{cor:Tesler-ineq} - the linear inequality description for $\tes_n(\ba)$ with all inequalities being facet-defining - to one for $\ptes_n(\ba),$ which provides us a facet normal vector for each facet $F_{i,j}(\ba)$ of $\ptes_{n}(\ba)$.

\begin{defn} \label{shiftedhs}
	    Let $\be_{i,j}^n \in \U(n)$ be the matrix in which the $(i,j)$-th entry is $1$ and all other entries are $0$.
	    For $1 \leq k \leq n$, define the \emph{$k$-th shifted hook sum matrix} to be
    \[\sm_k^{n}:=\sum_{j=k}^{n} \be_{k,j}^{n} - \sum_{i=1}^{k-1} \be_{i,k-1}^{n} \in \U(n).\]
\end{defn}
\begin{prop} \label{prop:projnormal}
	Let $\ba=(a_{1},\dots,a_{n}) \in \R_{> 0}^n$. 
    Then 
    \[ \ptes_n(\ba) = \biggl\{\YY \in \U(n-1)~~\biggl{|}~~
	    \begin{array}{lcl}
		    & \langle \sm_i^{n-1}, \YY \rangle \le a_i \text{ for all }1 \le i \le n-1\\[2mm]
		    & \langle -\be_{i,j-1}^{n-1}, \YY \rangle \le 0 \text{  for all }1 \leq i < j \leq n\\
 \end{array}
 \biggl\},\]
 where the inequality $\langle \sm_i^{n-1}, \YY \rangle \le a_i$ defines the facet $F_{i,i}(\ba)$ and the inequality $\langle -\be_{i,j-1}^{n-1}, \YY \rangle \le 0$ defines the facet $F_{i,j}(\ba).$

Hence, 
\begin{equation}\label{eq:defnnormal}
  \bn_{i,j}:=\begin{cases}
      \sm_i^{n-1}  & \text{if } {1 \le i=j \le n-1} \\
      -\be_{i,j-1}^{n-1} & \text{if } {1\le i<j \le n} 
    \end{cases} 
  \end{equation}
    is the primitive outer normal vector of the facet $F_{i,j}(\ba)$. 
\end{prop}

\begin{proof}
  By Corollary \ref{cor:Tesler-ineq} and Lemma \ref{projmap}, it is enough to show the following two statements are true:
  \begin{enumerate}[label=(\roman*)]
    \item For any $1 \le i \le n-1,$ the image of the halfspace $\{ \XX=(x_{l,k}) \in \H_n(\ba) \ | \ x_{i,i} \ge 0 \}$ of $\H_n(\ba)$ under $\psi_{\diag}$ is precisely the halfspace $\{ \YY \in \U(n-1) \ | \  \langle \sm_{i,}^{n-1}, \YY \rangle \le a_i\}$ of $\U(n-1)$.
    \item For any $1 \le i < j \le n,$ the image of the halfspace $\{ \XX=(x_{l,k}) \in \H_n(\ba) \ | \ x_{i,j} \ge 0 \}$ of $\H_n(\ba)$ under $\psi_{\diag}$ is precisely the halfspace $\{ \YY \in \U(n-1) \ | \  \langle -\be_{i,j-1}^{n-1}, \YY \rangle \le 0\}$ of $\U(n-1)$.
  \end{enumerate}
  We only proves (i) here, as the proof for (ii) is similar (and simpler). One sees that (i) is equivalent to that for any $1 \le i \le n-1$ and for any $\XX = (x_{l,k}) \in \H_n(\ba)$,
  \[ x_{i,i} \ge 0 \quad \Longleftrightarrow \quad \langle \sm_{i,}^{n-1}, \psi_{\diag}(\XX) \rangle \le a_i.\] 
  However, since $\XX \in \H_n(\ba)$, we have $\hs_i(\XX) = a_i,$ i.e.,
\[ (x_{i,i}+x_{i,i+1}+\cdots+x_{i,n}) - (x_{1,i}+x_{2,i}+\cdots+x_{i-1,i}) = a_i.\]
Hence,
\[ x_{i,i} \ge 0 \quad \Longleftrightarrow \quad  (x_{i,i+1}+\cdots+x_{i,n}) - (x_{1,i}+x_{2,i}+\cdots+x_{i-1,i}) \le a_i,\]
where one can verify that the latter is exactly $\langle \sm_{i,}^{n-1}, \psi_{\diag}(\XX) \rangle \le a_i.$
\end{proof}

\begin{cor}\label{cor:ptesnormalfan}
  For any $\ba, \bb \in \R_{>0}^n,$ the projected Tesler polytopes $\ptes_n(\ba)$ and $\ptes_n(\bb)$ have the same normal fan.
  \end{cor}

  \begin{proof}
    By the second part of Proposition \ref{prop:projnormal}, one sees that the normal fans of $\ptes_n(\ba)$ and $\ptes_n(\bb)$ have exactly the same one dimensional cones.
    It then follows from Lemma \ref{lem:construct-cones}/\eqref{itm:construct-ncone} and Proposition \ref{prop:ptes}/\eqref{itm:pface}, all other dimensional cones in these two normal fans are the same.
  \end{proof}

\subsection{BV-$\alpha$ values of projected Tesler polytopes} \label{subsec:bvptes}
By Part \eqref{itm:pint} of Proposition \ref{prop:ptes}, we know that for any $\ba \in \Z_{> 0}^n,$ the projected Tesler polytope $\ptes_n(\ba)$ is integral, and hence we can discuss the BV-$\alpha$ values arising from them. Moreover, it follows from Corollary \ref{cor:ptesnormalfan} that $\ptes_n(\ba)$ and $\ptes_n(\1)$ have the same normal fan, and thus by Lemma \ref{lem:nfcone} they share exactly the same BV-$\alpha$ values.  
For simplicity and convenience, below we will present our positivity results on BV-$\alpha$ values on $\ptes_n(\1)$ only, knowing that these results are true for any $\ptes_n(\ba)$ where $\ba \in \Z_{> 0}^n.$

First, we have seen in Proposition \ref{prop:projnormal} that any facet $F_{i,j}(\1)$ of $\ptes_n(\1)$ with $i<j$ has its outer normal as a member of the standard basis up to sign.
Then using \cite[Example 3.15]{Castillo2015v1}, we get the BV-$\alpha$ values for every face whose supporting facets are all of this form. 

\begin{lem} \label{cube}
	Suppose $F$ is a codimension $k$ face of $\ptes_{n}(\1)$, and its supporting facets are $F_{i_{1},j_{1}}(\1),\dots,F_{i_{k},j_{k}}(\1)$ where $i_{l}<j_{l}$ for all $1 \le l \le k.$ Then \[\bvalpha(F, \ptes_n(\1))=\dfrac{1}{2^k}.\]
\end{lem}

By Proposition \ref{prop:ptes}, the projected Tesler polytope $\ptes_{n}(\1)$ is totally unimodular, and is full-dimensional in $\U(n-1)$. Therefore, we can apply the procedure described in \S \ref{subsec:simplecomp} to calculate the BV-$\alpha$ values of codimension $2$ and codimension $3$ faces of $\ptes_{n}(\boldsymbol{1})$. 
Given Lemma \ref{cube}, we only consider faces of $\ptes(\1)$ where at least one of its supporting facets is $F_{l,l}(\ba)$ for some $1 \le l \le n-1$. We will calculate the BV-$\alpha$ values of faces we consider case by case, and use the following terminologies in our description of cases. 

Suppose $1 \le i < j \le n$ and $1 \le l \le n-1.$ We say that the position $(i,j)$ (of an upper triangular matrix) is \emph{on the $l$-th hook} if either $i=l$ or $j=l$. More specifically, we say $(i,j)$ is \emph{on the row of the $l$-th hook} if $i=l$ and $(i,j)$ is \emph{on the column of the $l$-th hook} if $j=l$. It is easy to verify that 
\[ \langle \bn_{l,l}, \bn_{i,j} \rangle =\left\langle \sm^{n-1}_{l}, -\be^{n-1}_{i,j-1} \right\rangle
  =\begin{cases}
    -1 & \quad \text{if $(i,j)$ is on the row of the $l$-th hook}; \\
    1 & \quad \text{if $(i,j)$ is on the column of the $l$-th hook}; \\
    0 & \quad \text{if $(i,j)$ is not on the $l$-th hook}.
  \end{cases}
\]
This result will be used repeatedly in calculations involved in the proof of Lemmas \ref{alphacd2} and \ref{alphacd3} below.

\begin{lem} \label{alphacd2} Assume $n \ge 3.$
Suppose $F$ is a codimension $2$ face of $\ptes_{n}(\boldsymbol{1})$, and its two supporting facets are $F_{l,l}(\1)$ and $F_{i,j}(\1)$ (with $i \le j$). We compute $\bvalpha(F,\ptes_{n}(\boldsymbol{1}))$ for all possible cases below.
\begin{enumerate}[leftmargin=*, label={\rm (\arabic*)}]
	\item Suppose $i <j$. There are three subcases.   
        \begin{enumerate}[label={\rm (\roman*)}]
            \item If $(i,j)$ is on the row of the $l$-th hook, then $\bvalpha(F,\ptes_{n}(\boldsymbol{1}))=\dfrac{1}{4}+\dfrac{1}{12}\biggl{(}\dfrac{n}{n-1}\biggl{)}$.
            \item If $(i,j)$ is on the column of the $l$-th hook, then $\bvalpha(F,\ptes_{n}(\boldsymbol{1}))=\dfrac{1}{4}-\dfrac{1}{12}\biggl{(}\dfrac{n}{n-1}\biggl{)}$.
            \item If $(i,j)$ is not on the $l$-th hook, then $\bvalpha(F,\ptes_{n}(\boldsymbol{1}))=\dfrac{1}{4}$.
        \end{enumerate}
    \item Suppose $i=j.$ Then $\bvalpha(F,\ptes_{n}(\boldsymbol{1}))=\dfrac{1}{4}+\dfrac{1}{6(n-1)}$.
\end{enumerate}
Hence, $\bvalpha(F,\ptes_{n}(\boldsymbol{1}))$ is positive for codimension $2$ faces of $\ptes_n(\1)$.
\end{lem}

\begin{proof}
  We only provide a proof for case (1)/(i), but one can obtain the BV-$\alpha$ values for the other cases by following the same procedure. For case (1)/(i), we have $i=l<j$. By Proposition \ref{prop:projnormal}, the primitive outer normal vector of $F_{l,l}(\1)$ is $\bn_{l,l}=\sm^{n-1}_{l}$ and that of $F_{i,j}(\1)$ is $\bn_{l,j} = -\be^{n-1}_{l,j-1}$. Thus, the MDP of $\ncone(F, \ptes_n(\boldsymbol{1}))$ is
	\[\begin{bmatrix}[1.2]
	    \left\langle \bn_{l,l}, \bn_{l,l} \right\rangle \ & \ \left\langle \bn_{l,l}, \bn_{l,j} \right\rangle \\
	    \left\langle \bn_{l,j}, \bn_{l,l} \right\rangle &  \left\langle \bn_{l,j}, \bn_{l,j} \right\rangle 
\end{bmatrix}
=
	  \begin{bmatrix}[1.2]
\left\langle \sm^{n-1}_{l}, \sm^{n-1}_{l} \right\rangle \ & \ \left\langle \sm^{n-1}_{l}, -\be^{n-1}_{l,j-1} \right\rangle \\
\left\langle -\be^{n-1}_{l,j-1}, \sm^{n-1}_{l} \right\rangle &  \left\langle -\be^{n-1}_{l,j-1}, -\be^{n-1}_{l,j-1} \right\rangle 
\end{bmatrix} = \begin{bmatrix}[1.2]
 n-1 & -1 \\
 -1 & 1 
\end{bmatrix}.\]
Then we compute its inverse, which by Corollary \ref{CMinv-restate} gives the MDP of $\fcone^p(F, \ptes_n(\boldsymbol{1}))$:
\[ \begin{bmatrix}[1.2]
 n-1 & -1 \\
 -1 & 1 
\end{bmatrix}^{-1} = \begin{bmatrix}[1.3] 
    \frac{1}{n-2} & \frac{1}{n-2} \\[1mm]
    \frac{1}{n-2} & \frac{n-1}{n-2} \\[1mm]
\end{bmatrix}.\]
Hence, if we let $\bu_1$ and $\bu_2$ be the primitive generators for the unimodular cone $\fcone^{p}(F,P)$ (with respect to the lattice $\Z^{\binom{n}{2}}/\lin(F)$) and assume $|\bu_1| \le |\bu_2|$, then  
\[\langle \bu_{1}, \bu_{1} \rangle = \dfrac{1}{n-2}, \quad  \langle \bu_{1}, \bu_{2} \rangle = \dfrac{1}{n-2},  \quad \text{ and } \quad \langle \bu_{2}, \bu_{2} \rangle = \dfrac{n-1}{n-2}.\]
Therefore, applying Lemma \ref{cd2}, we obtain the value of $\bvalpha(F,\ptes_{n}(\boldsymbol{1}))$ as shown in the lemma. 

As we mentioned above, we won't provide detailed calculation for the other cases. However, for easy reference, we summarize involved matrices and the $\alpha$-value of each case of this lemma in the table in Figure \ref{fig:data2}. 
\end{proof}

\begin{figure}[t]
 
\begin{tabular}{ | c | c | c | c |}
	\hline 
	\ Cases \	
	& \begin{tabular}{@{}c@{}}MDP of \\ $\ncone(F,\ptes_n(\1))$\end{tabular}
	& \begin{tabular}{@{}c@{}}MDP of \\ $\fcone^p(F,\ptes_n(\1))$\end{tabular}
	& $\alpha$-value 
	\rule{0pt}{1.5\normalbaselineskip} \\[0.7\normalbaselineskip]
 \hline
 (1)/(i)
 & $\begin{bmatrix}[1.2]
	 n-1 & -1 \\
 -1  & 1  
\end{bmatrix}$ & 

$\begin{bmatrix}[1.3]
 \frac{1}{n-2} & \frac{1}{n-2} \\[1mm]
 \frac{1}{n-2}  & \frac{n-1}{n-2}  \\[1mm]
\end{bmatrix}$ 
&
\ \  $\dfrac{1}{4}+\dfrac{1}{12} \biggl{(}\dfrac{n}{n-1}\biggl{)}$ 
\rule{0pt}{2\normalbaselineskip} 
 \\[1.3\normalbaselineskip]
\hline 

(1)/(ii) 
& $\begin{bmatrix}[1.2]
	n-1 & 1 \\
 1  & 1  
\end{bmatrix}$ &
$\begin{bmatrix}[1.3]
 \frac{1}{n-2} & \frac{-1}{n-2} \\[1mm]
 \frac{-1}{n-2}  & \frac{n-1}{n-2}  \\[1mm]
\end{bmatrix}$
&
\ \ $\dfrac{1}{4}-\dfrac{1}{12}\biggl{(}\dfrac{n}{n-1}\biggl{)}$
\rule{0pt}{2\normalbaselineskip} \\[1.3\normalbaselineskip]

\hline
(1)/(iii)
& $\begin{bmatrix}[1.2]
	n-1 &  0  \\
 0  & 1  \\
\end{bmatrix}$ &
$\displaystyle \begin{bmatrix}[1.3]
	\frac{1}{n-1} &  0  \\[1mm]
	0  & 1  
\end{bmatrix}$
&

\ \ $\dfrac{1}{4}$ 
\rule{0pt}{1.8\normalbaselineskip} \\[1.1\normalbaselineskip]
\hline

(2) 
& $\begin{bmatrix}[1.2]
	n-1 & -1 \\
 -1  & n-1  \\
\end{bmatrix}$ &

$\begin{bmatrix}[1.3]
 \frac{n-1}{n(n-2)} & \frac{1}{n(n-2)} \\[1mm]
 \frac{1}{n(n-2)}  & \frac{n-1}{n(n-2)}  \\[1mm]
\end{bmatrix}$
&

\ \ $\dfrac{1}{4}+\dfrac{1}{6(n-1)}$
\rule{0pt}{2\normalbaselineskip} \\[1.3\normalbaselineskip]
\hline

\end{tabular}

\caption{Data for Lemma \ref{alphacd2}}
\label{fig:data2}
\end{figure}

\begin{lem} \label{alphacd3}
  Assume $n \ge 3.$ Suppose $F$ is a codimension $3$ face of $\ptes_{n}(\boldsymbol{1})$, and its two supporting facets are $F_{l,l}(\1), F_{i_{1},j_{1}}(\1)$ and $F_{i_{2},j_{2}}(\1)$ (with $i_1 \le j_1$ and $i_2 \le j_2$). We compute $\bvalpha(F,\ptes_{n}(\boldsymbol{1}))$ for all possible cases below.
\begin{enumerate}[leftmargin=*, label={\rm (\arabic*)}]
    \item Suppose $i_{1} < j_{1}$ and $i_{2} < j_{2}$. There are six subcases.
    \begin{enumerate}[label={\rm (\roman*)}]
        \item If both $(i_{1},j_{1})$ and $(i_{2},j_{2})$ are not on the $l$-th hook, then $\bvalpha(F,\ptes_{n}(\boldsymbol{1}))=\dfrac{1}{8}$.
        \item If one of $(i_{1},j_{1})$ and $(i_{2},j_{2})$ is on the row of the $l$-th hook and the other is not on the $l$-th hook, then $\bvalpha(F,\ptes_{n}(\boldsymbol{1}))=\dfrac{1}{8} + \dfrac{n}{24(n-1)}$. 
	\item If one of $(i_{1},j_{1})$ and $(i_{2},j_{2})$ is on the column of the $l$-th hook and the other is not on the $l$-th hook, then $\bvalpha(F,\ptes_{n}(\boldsymbol{1}))=\dfrac{1}{8} - \dfrac{n}{24(n-1)}$.  
        \item If both of $(i_{1},j_{1})$ and $(i_{2},j_{2})$ are on the row of the $l$-th hook, then $\bvalpha(F,\ptes_{n}(\boldsymbol{1}))=\dfrac{1}{8} + \dfrac{n}{12(n-2)}$.
        \item If one of $(i_{1},j_{1})$ and $(i_{2},j_{2})$ is on the column of the $l$-th hook and the other is on the row of the $l$-th hook, then $\bvalpha(F,\ptes_{n}(\boldsymbol{1}))=\dfrac{1}{8} - \dfrac{1}{12(n-2)}$.
        \item If both of $(i_{1},j_{1})$ and $(i_{2},j_{2})$ are on the column of the $l$-th hook, then $\bvalpha(F,\ptes_{n}(\boldsymbol{1}))=\dfrac{1}{24}$.
    \end{enumerate}
    \item Suppose exactly one of $i_1=j_1$ and $i_2=j_2$ is true. Without loss of generality, we assume $i_1=j_1$ and let $m:=i_1.$ There are four subcases.
        \begin{enumerate}[label={\rm (\roman*)}]
		\item If $(i_{2},j_{2})$ is not on the $l$-th or the $m$-th hook, then $\bvalpha(F,\ptes_{n}(\boldsymbol{1}))=\dfrac{1}{8}+\dfrac{1}{12(n-1)}$.
            \item 
		   If $(i_{2},j_{2})$ is on the row of either the $l$-th hook or the $m$-th hook but is not on the hook of the other, 
		   then $\bvalpha(F,\ptes_{n}(\boldsymbol{1}))=\dfrac{1}{8}+\dfrac{n^{2}+n-3}{24(n^{2}-3n+2)}.$

            \item If $(i_{2},j_{2})$ is on the column of either the $l$-th hook or the $m$-th hook but is not on the hook of the other, 
	      then $\bvalpha(F,\ptes_{n}(\boldsymbol{1}))=\dfrac{1}{8} - \dfrac{n^2 - 3n + 3}{24(n^2-3n+2)}$. 

             \item If $(i_{2},j_{2})$ is on the column of either the $l$-th hook or the $m$-th hook and is on the row of the other hook, then $\bvalpha(F,\ptes_{n}(\boldsymbol{1}))=\dfrac{1}{8}$.
        \end{enumerate}
    \item Suppose $i_1=j_1$ and $i_2=j_2.$ Then $\bvalpha(F,\ptes_{n}(\boldsymbol{1}))=\dfrac{1}{8}+\dfrac{1}{4(n-2)}$.
    \end{enumerate}
Hence, $\bvalpha(F,\ptes_{n}(\boldsymbol{1}))$ is positive for codimension $3$ faces of $\ptes_n(\1)$.
\end{lem}

\begin{proof}
  We follow the same procedure as in the proof for case (1)/(i) of Lemma \ref{alphacd2} to obtain formulas for $\bvalpha(F,\ptes_{n}(\boldsymbol{1}))$ for all cases. Please refer to tables in Figure \ref{fig:data3} for all the matrices involved in the calculation. (We did not include $\alpha$-values in these tables so that they would fit into one single page.) We also remark that when $n=3$, cases (1)/(iv)(v)(vi) and case (3) won't occur, so $(n-3)$ occurring in the denominators of the MDP of $\fcone^p(F, \ptes_n(\1))$ in these cases is not a concern.
\end{proof}

\afterpage{
  \clearpage
\begin{landscape}
  \begin{figure} 

\begin{tabular}{ | c | c | c | }
	\hline 
	\ Cases \	
	& \begin{tabular}{@{}c@{}}MDP of \\ $\ncone(F,\ptes_n(\1))$\end{tabular}
	& \begin{tabular}{@{}c@{}}MDP of \\ $\fcone^p(F,\ptes_n(\1))$\end{tabular}

	\rule{0pt}{1.5\normalbaselineskip} \\[0.7\normalbaselineskip]
\hline
(1)/(i) & 
$\begin{bmatrix}[1.2]
n-1 & 0 & 0 \\
0  & 1  & 0  \\
0  & 0  & 1 
\end{bmatrix}$ & 
$\begin{bmatrix}[1.2]
  \frac{1}{n-1} & 0 & 0 \\[1mm]
0  & 1  & 0  \\
0  & 0  & 1  
\end{bmatrix} $ 

\rule{0pt}{2.2\normalbaselineskip} \\[1.5\normalbaselineskip]
\hline
(1)/(ii) & 
$\begin{bmatrix}[1.2]
n-1 & -1 & 0 \\
-1  & 1  & 0  \\
0  & 0  & 1  
\end{bmatrix}$ & 

$\begin{bmatrix}[1.2]
  \frac{1}{n-2} & \frac{1}{n-2} & 0 \\[1mm]
\frac{1}{n-2}  & \frac{n-1}{n-2}  & 0  \\[1mm]
0  & 0  & 1   
\end{bmatrix} $ 

\rule{0pt}{2.5\normalbaselineskip} \\[1.8\normalbaselineskip]
\hline
(1)/(iii) &
$\begin{bmatrix}[1.2]
n-1 & 1 & 0 \\
1  & 1  & 0  \\
0  & 0  & 1  
\end{bmatrix}$ & 
$\begin{bmatrix}[1.2]
  \frac{1}{n-2} & \frac{-1}{n-2} & 0 \\[1mm]
\frac{-1}{n-2}  & \frac{n-1}{n-2}  & 0  \\[1mm]
0  & 0  & 1  
\end{bmatrix} $ 

\rule{0pt}{2.5\normalbaselineskip} \\[1.8\normalbaselineskip]
\hline
(1)/(iv)&
$\begin{bmatrix}[1.2]
n-1 & -1 & -1 \\
-1  & 1  & 0  \\
-1  & 0  & 1  
\end{bmatrix}$ & 
$\begin{bmatrix}[1.2]
  \frac{1}{n-3} & \frac{1}{n-3} & \frac{1}{n-3} \\[1mm]
  \frac{1}{n-3}  & \frac{n-2}{n-3}  & \frac{1}{n-3}  \\[1mm]
  \frac{1}{n-3}  & \frac{1}{n-3}  & \frac{n-2}{n-3}  \\[1mm] 
\end{bmatrix} $ 

\rule{0pt}{2.5\normalbaselineskip} \\[1.8\normalbaselineskip]
\hline
(1)/(v)& 
$\begin{bmatrix}[1.2]
n-1 & -1 & 1 \\
-1  & 1  & 0  \\
1  & 0  & 1  
\end{bmatrix}$ & 
$\begin{bmatrix}[1.2]
  \frac{1}{n-3} & \frac{1}{n-3} & \frac{-1}{n-3} \\[1mm]
  \frac{1}{n-3}  & \frac{n-2}{n-3}  & \frac{-1}{n-3}  \\[1mm]
  \frac{-1}{n-3}  & \frac{-1}{n-3}  & \frac{n-2}{n-3}  \\[1mm] 
\end{bmatrix} $ 

\rule{0pt}{2.5\normalbaselineskip} \\[1.8\normalbaselineskip]
\hline
(1)/(vi) &
$\begin{bmatrix}[1.2]
n-1 & 1 & 1 \\
1  & 1  & 0  \\
1  & 0  & 1  
\end{bmatrix}$ & 
$\begin{bmatrix}[1.2]
  \frac{1}{n-3} & \frac{-1}{n-3} & \frac{-1}{n-3} \\[1mm]
  \frac{-1}{n-3}  & \frac{n-2}{n-3}  & \frac{1}{n-3}  \\[1mm]
  \frac{-1}{n-3}  & \frac{1}{n-3}  & \frac{n-2}{n-3} \\[1mm] 
\end{bmatrix} $ 

\rule{0pt}{2.5\normalbaselineskip} \\[1.8\normalbaselineskip]
\hline
\end{tabular}
\quad
\begin{tabular}{ | c | c | c |}
	\hline 
	\ Cases \	
	& \begin{tabular}{@{}c@{}}MDP of \\ $\ncone(F,\ptes_n(\1))$\end{tabular}
	& \begin{tabular}{@{}c@{}}MDP of \\ $\fcone^p(F,\ptes_n(\1))$\end{tabular}
	\rule{0pt}{1.5\normalbaselineskip} \\[0.7\normalbaselineskip]
\hline
(2)/(i) & 
$\begin{bmatrix}[1.2]
   n-1 & -1 & 0 \\
   -1  & n-1& 0\\
   0   & 0 & 1 
  \end{bmatrix}$& 
  $\begin{bmatrix}[1.2]
    \frac{n-1}{n(n-2)} & \frac{1}{n(n-2)} & 0\\[1mm]
    \frac{1}{n(n-2)} & \frac{n-1}{n(n-2)} & 0\\[1mm]
  0 & 0 & 1
  \end{bmatrix}$ 
  \rule{0pt}{2.5\normalbaselineskip} \\[1.8\normalbaselineskip]
\hline
(2)/(ii) &
$\begin{bmatrix}[1.2]
   n-1 & -1 & 0 \\
   -1  & n-1& -1\\
   0   & -1 & 1 
  \end{bmatrix}$& 
  $\begin{bmatrix}[1.2]
    \frac{n-2}{n^{2} -3n +1} & \frac{1}{n^{2} -3n +1} & \frac{1}{n^{2} -3n +1}\\[1mm]
    \frac{1}{n^{2} -3n +1}   & \frac{n-1}{n^{2} -3n +1} & \frac{n-1}{n^{2} -3n +1}\\[1mm]
    \frac{1}{n^{2} -3n +1}& \frac{n-1}{n^{2} -3n +1} & \frac{n^{2}-2n}{n^{2} -3n +1}\\[1mm]
  \end{bmatrix}$
  \rule{0pt}{2.5\normalbaselineskip} \\[1.8\normalbaselineskip]
\hline
(2)/(iii) &
$\begin{bmatrix}[1.2]
   n-1 & -1 & 0 \\
   -1  & n-1& 1\\
   0   & 1 & 1 
  \end{bmatrix}$& 
  $\begin{bmatrix}[1.2]
    \frac{n-2}{n^{2} -3n +1} & \frac{1}{n^{2} -3n +1} & \frac{-1}{n^{2} -3n +1}\\[1mm]
    \frac{1}{n^{2} -3n +1}   & \frac{n-1}{n^{2} -3n +1} & \frac{-n+1}{n^{2} -3n +1}\\[1mm]
    \frac{-1}{n^{2} -3n +1}& \frac{-n+1}{n^{2} -3n +1} & \frac{n^{2}-2n}{n^{2} -3n +1}\\[1mm]
  \end{bmatrix}$
  \rule{0pt}{2.5\normalbaselineskip} \\[1.8\normalbaselineskip]
\hline

(2)/(iv) &
$\begin{bmatrix}[1.2]
   n-1 & -1 & -1 \\
   -1  & n-1& 1\\
   -1   & 1 & 1 
  \end{bmatrix}$& 
  $\begin{bmatrix}[1.2]
    \frac{1}{n-2} & 0 & \frac{1}{n-2}\\[1mm]
    0  & \frac{1}{n-2} & \frac{-1}{n-2}\\[1mm]
    \frac{1}{n-2} & \frac{-1}{n-2} & \frac{n}{n-2} \\[1mm]
  \end{bmatrix}$ 
  \rule{0pt}{2.5\normalbaselineskip} \\[1.8\normalbaselineskip]
\hline
(3) & 
$\begin{bmatrix}[1.2]
  n-1 & -1 & -1\\
  -1  & n-1& -1\\
  -1  & -1 &n-1
 \end{bmatrix}$ &
 $\begin{bmatrix}[1.2]
  \frac{n-2}{n(n-3)} & \frac{1}{n(n-3)} & \frac{1}{n(n-3)}\\[1mm]
  \frac{1}{n(n-3)} & \frac{n-2}{n(n-3)} & \frac{1}{n(n-3)}\\[1mm]
  \frac{1}{n(n-3)} & \frac{1}{n(n-3)} & \frac{n-2}{n(n-3)}\\[1mm]
 \end{bmatrix}$ 
 \rule{0pt}{2.5\normalbaselineskip} \\[1.8\normalbaselineskip]
\hline
\end{tabular}
\caption{Matrices for Lemma \ref{alphacd3}}
\label{fig:data3}
\end{figure}
\end{landscape}
}

We are now ready to prove Theorem \ref{ptes-positive}.

\begin{proof}[Proof of Theorem \ref{ptes-positive}]
The first conclusion follows from Lemmas \ref{alphacd2} and \ref{alphacd3}.
The second conclusion follows from the first conclusion and Equation (\ref{maccoeffi}). 
\end{proof}

\subsection{Corollaries to Theorem \ref{ptes-positive}}

Using the Reduction Theorem (Theorem \ref{reduction}), we can generalize our BV-$\alpha$-positivity results for codimension $2$ and $3$ faces of $\ptes_{n}(\boldsymbol{1})$ to any deformations of $\ptes_{n}(\boldsymbol{1})$:

\begin{cor} \label{maincor}
Let $P$ be a $d$-dimensional integral polytope that is a deformation of $\ptes_{n}(\boldsymbol{1})$. Then the following statements are true:
\begin{enumerate}[label={\rm (\arabic*)}]
  \item If $d=\dim(\ptes_{n}(\boldsymbol{1}))=\binom{n}{2}$, then all the codimension $2$ and $3$ faces of $P$ have positive BV-$\alpha$ values. Therefore, $e_{d-2}(P)>0$ and $e_{d-3}(P)>0$.
  \item If $d=\dim(\ptes_{n}(\boldsymbol{1}))-1=\binom{n}{2}-1$, then all the codimension $2$ faces of $P$ have positive BV-$\alpha$ values. Therefore, $e_{d-2}(P)>0$.
\end{enumerate}
\end{cor}

\begin{proof}If $d=\dim(\ptes_{n}(\boldsymbol{1}))$, then by Lemmas \ref{alphacd2} and \ref{alphacd3} and Theorem \ref{reduction}, we have that $\bvalpha(F,P)>0$ for every face $F$ of $P$ that has dimension $d-2$ or $d-3$. Therefore, the coefficients $e_{d-2}(P)$ and $e_{d-3}(P)$ are positive by Equation (\ref{maccoeffi}). 
	
	Similarly, if $d=\dim(\ptes_{n}(\boldsymbol{1}))-1$, Lemma \ref{alphacd3} and Theorem \ref{reduction} imply that $\bvalpha(F,P)>0$ for all $(d-2)$-dimensional faces $F$ of $P$. Therefore, we have that $e_{d-2}(P)>0$ by Equation (\ref{maccoeffi}).
\end{proof}

\begin{proof}[Proof of Corollary \ref{teslerdef-positiv}]
	Recall that $\ptes_{n}(\boldsymbol{1}) = \psi_{\diag}\left( \tes_n(\boldsymbol{1})  \right)$, where by Lemma \ref{projmap} the map $\psi_{\diag}$ induces a unimodular transformation from $\H_n(\boldsymbol{1})$ to $\U(n-1).$ Then the conclusions follow from Corollary \ref{maincor}, Lemma \ref{unimodulardef}, and the fact that unimodular transformations are invertible and preserve Ehrhart polynomials.
\end{proof}

\begin{proof}[Proof of Corollary \ref{cor:tesler}] Suppose the first $p$ entries of $\ba$ are zero and $a_{p+1} > 0,$ and let $\bb = (a_{p+1}, a_{p+2}, \dots, a_n)$. By Lemma \ref{lem:Tes-dim}, the Tesler polytope $\tes_n(\ba)$ is isomorphic to the Tesler polytope $\tes_{n-p}(\bb)$ which has the same dimension as $\tes_{n-p}(\1).$ It is a classical result that $\tes_{n-p}(\bb)$ is a deformation of $\tes_{n-p}(\1).$ (See \cite{de2010triangulations} or \cite[Section 4]{teslerv1}.)
  Therefore, the conclusion follows from Part (1) of Corollary \ref{teslerdef-positiv}.
\end{proof}

\bibliographystyle{abbrv}
\bibliography{ref}

\end{document}